\documentclass[12pt,reqno]{amsart}
\usepackage{graphicx}
\usepackage[utf8]{inputenc}
\usepackage{amsmath}
\usepackage{amsfonts}
\usepackage{amssymb}
\usepackage{enumitem}
\numberwithin{equation}{section}
\theoremstyle{plain}
\newtheorem{thm}{Theorem}
\newtheorem{proposition}{Proposition}
\newtheorem{corollary}{Corollary}
\newtheorem{lemma}{Lemma}
\newtheorem{assumption}{Assumption}

\theoremstyle{definition}

\theoremstyle{remark}
\newtheorem{remark}{Remark}
\newtheorem{example}{Example}

\begin{document}
\title[Variable selection based on norms of projections]{Variable selection in high-dimensional additive models based on norms of projections}
\author{Martin Wahl}
\address{Institut f\"{u}r Angewandte Mathematik, Universit\"{a}t Heidelberg,
              Im Neuenheimer Feld 294, 69120 Heidelberg, Germany}
\email{wahl@uni-heidelberg.de} 
\begin{abstract}
We consider the problem of variable selection in high-dimensional sparse additive models. We focus on the case that the
components belong to nonparametric classes of functions. The proposed method is motivated by geometric considerations in Hilbert spaces and consists of comparing the norms of the projections of the data onto various additive subspaces. Under minimal geometric assumptions, we prove concentration inequalities which lead to new conditions under which consistent variable selection is possible. As an application, we establish conditions under which a single component can be estimated with the rate of convergence corresponding to the situation in which the other components are known.
\end{abstract}

\maketitle
\section{Introduction}
In this paper, we consider the problem of variable selection in high-dimensional nonparametric additive models in which the number of covariates is much larger than the number of observations. We study these models under the assumption that most components are equal to zero.

High-dimensional linear models have been investigated intensively in the literature. A great deal of attention has been given to the Lasso (see, e.g., the book by B\"{u}hlmann and van de Geer \cite{BvG} and the references therein). The Lasso is based on $l_1$-penalization, and can be used for both estimation and variable selection. There is also a huge literature on estimation and variable selection via $l_0$-penalization. These procedures can be found, e.g., in the book by Massart \cite{M}, where a general approach to model selection via penalization is developed (see also the work by Barron, Birg\'{e}, and Massart \cite{BBM} and the references therein). Finally, there is a third approach which is based on exponentially weighted aggregation (see, e.g., the work by Rigollet and Tsybakov \cite{RTs} and Arias-Castro and Lounici \cite{ACL} and the references therein).

More recently, high-dimensional additive models have been studied, e.g., in the work by Meier, van de Geer, and B\"{u}hlmann \cite{MGB}, Huang, Horowitz, and Wei \cite{HHW}, Koltchinskii and Yuan \cite{KY}, Raskutti, Wainwright, and Yu \cite{RWY}, Gayraud and Ingster \cite{GI}, Suzuki and Sugiyama \cite{SuSu}, and Dalalyan, Ingster, and Tsybakov \cite{DIT}. 
One approach generalizes the (group) Lasso and combines sparsity penalties with smoothness penalties or constraints (see \cite{MGB,HHW,KY,RWY,SuSu}). As in the case of the Lasso, these procedures can be used for both estimation and variable selection (see \cite{MGB,HHW}). Another approach based on exponential aggregation is developed in the work by Dalalyan, Ingster, and Tsybakov \cite{DIT}. They considered the problem of estimation in a more general model which they called the compound model and which includes the additive model as a special case.
In a Gaussian white noise setting, they showed that their estimator achieves non-asymptotic minimax rates of convergence.

Comminges and Dalalyan \cite{CD} considered the problem of variable selection in a high-dimensional Gaussian white noise model, and established tight conditions which make the estimation of the relevant variables possible.
They also extended their method to a high-dimensional random regression model, but they assumed that the joint density of all covariates is known.
Similar results were obtained earlier by Wainwright \cite{Wa} for high-dimensional linear models with Gaussian measurement matrices.

Several results in the theory of high-dimensional statistical inference are initiated by achievements in the theory of compressive sensing (see, e.g., the introductory book chapters by Fornasier and Rauhut \cite{ra10} and Rauhut \cite{R} and the references therein). A popular method is the $l_1$-minimization which enables sparse recovery if the measurement matrix satisfies, for instance, a restricted isometry property (RIP). It is known that several random matrices satisfy the RIP with probability close to one,  important examples being the Gaussian random matrices and the so-called structured random matrices (see, e.g., the work by Cand\`{e}s and Tao \cite{CT2}, Baraniuk, Davenport, DeVore, and Walkin \cite{BDDW}, and Rauhut \cite{R}). These results were generalized to high-dimensional linear models by Cand\`{e}s and Tao \cite{CT} (see also the work by Bickel, Ritov, and Tsybakov \cite{BRT} and the book by Koltchinskii \cite[Chapters 7 and 8]{Ko}).

In this paper, we study a method for variable selection which consists of comparing the norms of the projections of the data onto various finite-dimensional additive subspaces. Given an upper bound $q^*$ for the number of nonzero components, the procedure selects the subset of cardinality smaller than or equal to $q^*$ which best explains the data in the finite sample setting. The basis of this procedure is a selection criterion in the population setting which works well under the essential assumption that the minimal angles between various disjoint additive subspaces are bounded away from zero.
Applying this assumption and tools from the theory of structured random matrices, we derive a strong uniform concentration property of the empirical norm around the $L^2(\mathbb{P}^X)$-norm, which, in the special case of independent covariates, can be rewritten as a restricted (block)-isometry property. This property enables us to carry over the geometry in the population setting to the finite sample setting, and thus leads to an analysis of our procedure. 
Our results are of theoretical interest. Under minimal geometric assumptions, we prove upper bounds for the probability that our procedure misses relevant variables. These concentration inequalities lead to conditions making consistent estimation of the relevant variables possible. In the case of the linear model with random measurement and also in settings considered in the theory of compressive sensing, these conditions coincide with what can be usually found in the literature (see, e.g., \cite{Wa,R}).
In the general case of the nonparametric additive model, we find conditions which are, to the best of our knowledge, new.
As an application of our variable selection procedure, we consider the problem of estimating single components.
We establish conditions under which a single component can be estimated with the rate of convergence corresponding to the situation in which the other components are known. This is of interest since the rate of convergence valid for the whole regression function can be considerably smaller if the number of nonzero components is large or if the component of interest is smoother than the other components.

The paper is organized as follows. In Section \ref{mrs}, we present the main assumption and discuss a selection criterion in the population setting. Moreover, we propose our selection criterion and state a first version of our main result in Theorem \ref{mta}.
In Section \ref{grip}, we complete our main result by proving a uniform concentration property of the empirical norm. Moreover, we apply our result to the problem of estimating a single component. Section \ref{opt} is devoted to the analysis of the finite sample geometry, which is the main part in the proof of Theorem \ref{mta}. Finally, some technical parts of the proofs are given in the Appendix.
\section{The main result}\label{mrs}
\subsection{The variable selection problem}
Let $(Y,X)$ be a pair of random variables such that $X=(X_1,\dots,X_q)^T$ and
\begin{equation}\label{eq:sam}
Y=\sum_{j=1}^qf_j(X_j)+\epsilon,
\end{equation}
where the $X_j$ are real-valued random variables, the $f_j$ are unknown functions which are contained in $L^2(\mathbb{P}^{X_j})$, and $\epsilon$ is a Gaussian random variable with expectation $0$ and variance $\sigma^2$ which is independent of $X$. Moreover, we suppose that $f_j$ satisfies $\mathbb{E}[f_j(X_j)]=0$ for $j=1,\dots,q-1$. We denote by $f$ the whole regression function given by $f(x)=\sum_{j=1}^qf_j(x_j)$. 
We assume that we observe $n$ independent copies $(Y^1,X^1),\dots,(Y^n,X^n)$ of $(Y,X)$, i.e.,
\begin{equation}
Y^i=\sum_{j=1}^qf_j(X^i_j)+\epsilon^i,\ \ \ i=1,\dots,n.
\end{equation}  
The number of covariates $q$ can be much larger than the number of observations $n$, but we assume that the number of non-zero components is smaller than $n$. Thus we consider a high-dimensional sparse additive model. We define
$J_0=\left\{j\in\{1,\dots,q\}:\|f_j\|>0\right\}$, meaning that we have $f(x)=\sum_{j\in J_0}f_j(x_j)$. Moreover, we denote by $s$ the cardinality of $J_0$, i.e., $s=|J_0|$. The set $J_0$ is supposed to be unknown, but we assume that we are given an integer $q^*$ such that $|J_0|\leq q^*$. We aim at selecting a subset of cardinality smaller than or equal to $q^*$ which contains $J_0$.

\subsection{The main assumption} Without any further assumption, the components are not necessarily uniquely determined. In this section, we give an assumption which implies uniqueness and furthermore makes the variable selection task accessible. We define $H_q=L^2(\mathbb{P}^{X_q})$ and  
\[H_j=\left\{h_j\in L^2(\mathbb{P}^{X_j})|\mathbb{E}\left[h_j(X_j)\right]=0\right\}\]
for $j=1,\dots,q-1$.
Note that $f_j\in H_j$. The spaces $H_j$ are all canonically contained in $L^2(\mathbb{P}^X)$ which is a Hilbert space with the inner product $\langle g, h\rangle=\mathbb{E}[g(X)h(X)]$ and the corresponding norm $\|g\|=\sqrt{\langle g, g\rangle}$. Moreover, for $J\subseteq\lbrace1,\dots,q\rbrace$, we define
\[H_J=\sum_{j\in J}H_j\]
(with the convention that $H_J=0$ if $J=\emptyset$). 
\begin{assumption}\label{anglesep}
There exists a constant $0\leq\rho <1$ such that for all subsets $J_1,J_2\subseteq\lbrace1,\dots,q\rbrace$ satisfying $J_1\cap J_2=\emptyset$ and $\vert J_1\vert, \vert J_2\vert\leq q^*$, we have
\begin{equation}\label{eq:anglesepeq}
\langle h_{J_1},h_{J_2}\rangle\leq \rho \Vert h_{J_1}\Vert\Vert h_{J_2}\Vert
\end{equation}
for all $h_{J_1}\in H_{J_1}$, $h_{J_2}\in H_{J_2}$. 
\end{assumption} 
It follows from the fact that the spaces $H_j$ are closed combined with Assumption \ref{anglesep} and \cite[Theorem 1a]{K} (applied inductively) that all spaces $H_J$ with $J\subseteq\lbrace1,\dots,q\rbrace$ and $|J|\leq 2q^*$ are closed. The real number
\begin{equation*}
 \rho_0(H_{J_1},H_{J_2})=\sup\left\lbrace\frac{\langle h_{J_1},h_{J_2}\rangle}{\Vert h_{J_1}\Vert\Vert h_{J_2}\Vert}\bigg|0\neq h_{J_1}\in H_{J_1},0\neq h_{J_2}
 \in H_{J_2}\right\rbrace
\end{equation*}
is the cosine of the minimal angle between $H_{J_1}$ and $H_{J_2}$ (see, e.g., \cite[Definition 1]{KW}). Letting $\rho_{q^*}=\max\rho_0(H_{J_1},H_{J_2})$, where the maximum is taken over all subsets $J_1,J_2\subseteq\lbrace1,\dots,q\rbrace$ satisfying $J_1\cap J_2=\emptyset$ and $\vert J_1\vert, \vert J_2\vert\leq q^*$, then Assumption \ref{anglesep} says that $\rho_{q^*}<1$. By a simple argument which is given in Appendix \ref{efa}, one can show that Assumption \ref{anglesep} can be written as follows:

\begin{remark}[Equivalent form of Assumption \ref{anglesep}]\label{angeqsin}
For all subsets $J_1,J_2\subseteq\lbrace1,\dots,q\rbrace$ satisfying $J_1\cap J_2=\emptyset$ and $\vert J_1\vert, \vert J_2\vert\leq q^*$, we have
\begin{equation}\label{eq:mmm}
\left\| h_{J_1}+h_{J_2}\right\|^2\geq (1-\rho_{q^*}^2)\left\| h_{J_1}\right\|^2
\end{equation}
for all $h_{J_1}\in H_{J_1}$, $h_{J_2}\in H_{J_2}$.
\end{remark}
Remark \ref{angeqsin} shows that Assumption \ref{anglesep} is essential for variable selection: if \eqref{eq:mmm} does not hold, then it is possible that $f$ is arbitrary close to a sparse additive function which is based on a completely different set of variables.
From \eqref{eq:mmm} and the definition of $J_0$, we obtain:
\begin{lemma}\label{minlemma} Let Assumption \ref{anglesep} be satisfied. Then
\[\kappa:=\min_{\emptyset\neq J\subseteq J_0}\Big\Vert \sum_{j\in J}f_j\Big\Vert^2>0.\]
\end{lemma}
For $J\subseteq\lbrace1,\dots,q\rbrace$ let $\Pi_{H_J}$ be the orthogonal projection from $L^2(\mathbb{P}^X)$ to $H_J$. In the following we abbreviate $\Pi_{H_J}$ as $\Pi_{J}$.
Since projections lower the norm, the set $J_0$ maximizes the quantity $\left\Vert \Pi_{J}f\right\Vert^2$. If Assumption \ref{anglesep} holds, the following Lemma shows that $\left\Vert \Pi_{J_0}f\right\Vert^2-\left\Vert \Pi_{J}f\right\Vert^2$ is strictly positive for all subsets $J\subseteq\lbrace1,\dots,q\rbrace$ with $\vert J\vert\leq q^*$ and $J_0\setminus J\neq\emptyset$. This means that a subset $J\subseteq\lbrace1,\dots,q\rbrace$ with $\vert J\vert\leq q^*$ which maximizes $\left\Vert \Pi_{J}f\right\Vert^2$ always contains $J_0$ (and is  equal to $J_0$ in the special case when $|J_0|=q^*$). These observations will be the theoretical basis for our selection criterion in the finite sample setting.  
\begin{proposition}\label{fundlem} 
Let Assumption \ref{anglesep} be satisfied, and let $J\subseteq\lbrace1,\dots,q\rbrace$ be a subset such that $\vert J\vert\leq q^*$ and $J_0\setminus J\neq\emptyset$. Then 
\[\left\Vert \Pi_{J_0}f\right\Vert ^2-\left\Vert \Pi_{J}f\right\Vert ^2=\left\|f-\Pi_{J}f\right\|^2\geq(1-\rho_{q^*}^2)\kappa_l,\] 
where $l=|J_0\setminus J|$ and
\[\kappa_l:=\min_{J'\subseteq J_0,|J'|=l}\Big\Vert \sum_{j\in J'}f_j\Big\Vert^2.\]
\end{proposition}

\begin{proof} 
The equality follows from $\Pi_{J_0}f=f$ and the projection theorem. We turn to the proof of the inequality.
We have $f=\sum_{j\in J_0\cap J}f_j+\sum_{j\in J_0\setminus J}f_j=:f_{J_0\cap J}+f_{J_0\setminus J}$. Hence
\[\Pi_{J}f=f_{J_0\cap J}+\Pi_Jf_{J_0\setminus J}\]
and 
\[f-\Pi_{J}f=f_{J_0\setminus J}-\Pi_Jf_{J_0\setminus J}.\] 
We have $f_{J_0\setminus J}\in H_{J_0\setminus J}$, $\Pi_Jf_{J_0\setminus J}\in H_J$, and $l=|J_0\setminus J|\geq 1$. Thus \eqref{eq:mmm} and the definition of $\kappa_l$ yield 
\[\left\|f-\Pi_{J}f\right\|^2=\left\Vert f_{J_0\setminus J}-\Pi_Jf_{J_0\setminus J}\right\Vert ^2\geq (1-\rho_{q^*}^2)\Vert f_{J_0\setminus J}\Vert^2\geq (1-\rho_{q^*}^2)\kappa_l.\] 
This completes the proof.
\end{proof}
Finally, we show that $\rho_{q^*}$ can be related to a quantity which is known in the literature on sparse additive models (see, e.g., \cite{KY}).
\begin{lemma}\label{ric}
Let $\epsilon_{2q^*}$ be the smallest number such that 
\begin{equation}\label{eq:riceq} 
\Big\|\sum_{j\in J}f _j\Big\|^2\geq(1-\epsilon_{2q^*})\Big(\sum_{j\in J}\left\|f _j\right\|^2\Big)
\end{equation}
for all $J\subseteq\left\{1,\dots,q\right\}$ with $|J|\leq {2q^*}$ and all $\sum_{j\in J}f_j\in H_J$. Then we have $\rho_{q^*}<1$ if and only if $\epsilon_{2q^*}<1$.
\end{lemma}
A proof of this lemma is given in Appendix \ref{eas}. 

\subsection{The selection criterion}\label{selcrit}
In this section, we construct the selection criterion.
For $j=1,\dots,q$, let $V_j\subseteq H_j$ be finite-dimensional linear subspaces. For $J\subseteq\lbrace1,\dots,q\rbrace$, let \[V_J=\sum_{j\in J}V_j\]
and $d_J=\dim V_J$. Moreover, for $l=1,\dots,q$, let $d_{l}=\max_{|J|=l}d_J$.

In order to proceed, we introduce some further notation. Let $\|\cdot\|_n$ be the empirical norm which is defined by 
\[\|h\|_n^2=\frac{1}{n}\sum_{i=1}^nh^2(X^i)\]
for $h\in L^2(\mathbb{P}^X)$, and which is defined by $\|\cdot\|_n^2=(1/n)\|\cdot\|_2^2$ if applied to vectors in $\mathbb{R}^n$. Here, $\|\cdot\|_2$ denotes the usual Euclidean norm. Moreover, let $\hat{\Pi}_{J}$ be the orthogonal projection from $\mathbb{R}^n$ to the subspace $\lbrace (g_J(X^1),\dots,g_J(X^n))^T|g_J\in V_J\rbrace$. If $h\in L^2(\mathbb{P}^X)$, then we abbreviate $\hat{\Pi}_J(h(X^1),\dots,h(X^n))^T$ as $\hat{\Pi}_Jh$.
Finally, let $\mathbf{Y}=(Y^1,\dots,Y^n)^T$ and $\boldsymbol{\epsilon}=(\epsilon^1,\dots,\epsilon^n)^T$. 
Motivated by Proposition \ref{fundlem}, we define an estimator $\hat{J}_0$ of $J_0$ as follows: 
\begin{equation}\label{eq:sc}
\hat{J_0}=\operatorname{arg\ max}_{J\subseteq\lbrace1,\dots,q\rbrace, |J|\leq q^*}\limits\left(\big\Vert\hat{\Pi}_J\mathbf{Y}\big\Vert^2_n-\sigma^2d_J/n\right).
\end{equation}
Conditioning on $X^1,\dots,X^n$, the random variable $(n/\sigma^2)\Vert\hat{\Pi}_J\boldsymbol{\epsilon}\Vert_n^2$ has a chi-square distribution with $\operatorname{rank}(\hat{\Pi}_J)\leq d_J$ degrees of freedom and the last term is supposed to cancel its expectation. The last term can also be seen as a penalty term. In fact, the criterion in \eqref{eq:sc} can be written as a penalized least squares criterion (see, e.g., \cite{M}). 

The success of the criterion depends on a suitable choice of the $V_j$, which in turn depends on the regularity conditions of the $f_j$. For instance, if the $f_j$ belong to some known finite-dimensional linear subspaces of $H_j$, then we let the $V_j$ be equal to these spaces. In the following, we consider the nonparametric case.
Without loss of generality, we shall restrict our attention to (periodic) Sobolev smoothness and spaces of trigonometric polynomials. 
A similar treatment is possible, e.g., for H\"{o}lder smoothness and spaces of piecewise polynomials or spaces of splines.
Recall that the trigonometric basis is given by $\phi_1(x)= 1$ and $\phi_{2k}(x)=\sqrt{2}\cos(2\pi kx)$ and $\phi_{2k+1}(x)=\sqrt{2}\sin(2\pi kx)$, $k\geq 1$, where $x\in[0,1]$.
\begin{assumption}\label{hass}
Suppose that the $X_j$ take values in $[0,1]$ and have densities $p_j$ with respect to the Lebesgue measure on $[0,1]$, which satisfy $c\leq p_j\leq 1/c$ for some constant $c>0$. Moreover, suppose that the $f_j$ belong to the Sobolev classes 
\[\tilde{W}_j(\alpha_j,K_j)=\left\{\sum_{k=1}^\infty \theta_k\phi_k(x_j)\ :\ \sum_{k=1}^\infty(2\pi k)^{2\alpha_j}(\theta_{2k}^2+\theta_{2k+1}^2)\leq K_j^2\right\},\]
where $\alpha_j>1/2$ and $K_j> 0$ (see, e.g., \cite[Definition 1.12]{T}).
\end{assumption}
For $j=1,\dots,q$, let $V_j$ be the intersection of $H_j$ with the linear span of $\phi_1,\dots,\phi_{m_j}$ (in the variable $x_j$). The choice of the $m_j$ will depend on the following approximation properties. 
\begin{lemma}\label{lemappr} Let Assumption \ref{hass} be satisfied. Then there exists a constant $C_j>0$ depending only on $\alpha_j$ and $c$ (given explicitly in the proof) such that
\begin{align*}
&\|h_j-\Pi_{V_j}h_j\|^2\leq C_jK_j^2m_j^{-2\alpha_j}\ \ \ \text{ and}\\
& \|h_j-\Pi_{V_j}h_j\|_\infty^2\leq C_jK_j^2m_j^{1-2\alpha_j}
\end{align*}
for all $h_j\in\tilde{W}_j(\alpha_j,K_j)\cap H_j$, where $\Pi_{V_j}$ is the orthogonal projection from $L^2(\mathbb{P}^X)$ to $V_j$.
\end{lemma}
For completeness, a proof of this lemma is given in Appendix \ref{plemappr}. We suppose that for $j=1,\dots,q$,
\begin{equation}\label{eq:lowbound}
 m_j\geq\left(\frac{C_jK_j^2q^*(1+\epsilon_{q^*}')}{c'(1-\rho_{q^*}^2)\kappa}\right)^{1/2\alpha_j},
\end{equation}
where $0<c'<1$ is a small constant satisfying \eqref{eq:ccond} and $\epsilon_{q^*}'$ is a positive real number such that
\begin{equation}\label{eq:ric2}
\Big\|\sum_{j\in J}f _j\Big\|^2\leq(1+\epsilon_{q^*}')\Big(\sum_{j\in J}\left\|f _j\right\|^2\Big)
\end{equation}
for all $J\subseteq\left\{1,\dots,q\right\}$ with $|J|\leq q^*$ and all $\sum_{j\in J}f_j\in H_J$. Note that, by the Cauchy-Schwarz inequality, we can always choose $1+\epsilon_{q^*}'=q^*$. The $m_j$ are chosen such that the following upper bound holds
\begin{equation}\label{eq:af2}
\Big\|f-\sum_{j\in J_0}\Pi_{V_j}f_j\Big\|^2\leq (1+\epsilon_{q^*}')\sum_{j\in J_0}\|f_j-\Pi_{V_j}f_j\|^2\leq c'(1-\rho_{q^*}^2)\kappa,
\end{equation}
where we used \eqref{eq:ric2} and Lemma \ref{lemappr}. 
Applying Bennett's inequality and Lemma \ref{lemappr}, one can show that a similar bound holds with high probability when the $L^2(\mathbb{P}^X)$-norm is replaced by the empirical norm $\|\cdot\|_n$. The result is as follows:
\begin{lemma}\label{testdev2} Let Assumptions \ref{anglesep} and \ref{hass} be satisfied. Suppose that \eqref{eq:lowbound} is satisfied for $j=1,\dots,q$. Let the event $\mathcal{A}$ be given by
\[\mathcal{A}=\left\{\Big\|f-\sum_{j\in J_0}\Pi_{V_j}f_j\Big\|^2_n\leq 2c'(1-\rho_{q^*}^2)\kappa\right\}.\]
Then
\begin{equation}\label{eq:mn}
\mathbb{P}\left(\mathcal{A}^c\right)\leq \exp\left(-\frac{3}{16}\frac{n}{d_{q^*}}\right).
\end{equation}
\end{lemma}
A proof of Lemma \ref{testdev2} is given in Appendix \ref{adev2}
\subsection{The main result and some consequences}\label{tmr} In this section, we present our first main theorem and derive several consequences. These results will be further developed in Section \ref{grip}, where the final results can be found.

For $J\subseteq\lbrace 1,\dots,q\rbrace$ and $0<\delta<1$ (e.g. $\delta=1/2$), we define the events
\begin{equation*}\mathcal{E}_{\delta,J}=\left\lbrace (1-\delta)\Vert g_J\Vert^2\leq\Vert g_J\Vert_n^2\leq (1+\delta)\Vert g_J\Vert^2\ \text{ for all }g_J\in V_J\right\rbrace.
\end{equation*}
Moreover, we define
\begin{equation*}
\mathcal{E}_{\delta,q^*}=\bigcap_{J\subseteq\{1,\dots,q\},|J|\leq q^*}\mathcal{E}_{\delta,J\cup J_0}.
\end{equation*}
We prove:
\begin{thm}\label{mta} Let Assumptions \ref{anglesep} and \ref{hass} be satisfied. Let $0<\delta<1$. Suppose that \eqref{eq:lowbound} is satisfied for $j=1,\dots,q$. Then there is a constant $c_1>0$ depending only on $\delta$ (given explicitly in the proof) such that
\begin{align}
&\mathbb{P}\left( J_0\subseteq \hat{J}_0 \right)\nonumber\geq 1-\mathbb{P}\left(\mathcal{E}_{\delta,q^*}^c\right)-\exp\left(-\frac{3}{16}\frac{n}{d_{q^*}}\right)\nonumber\\
&-\sum_{l=1}^{s}\sum_{m=0}^{q^*-(s-l)}\binom{s}{l}\binom{q-s}{m}
4\exp\left( -c_1\frac{n^2(1-\rho_{q^*}^2)^2\kappa_l^2}{\sigma^4d_{q^*-s+l} +\sigma^2n(1-\rho_{q^*}^2)\kappa_l}\right).\label{eq:mres}
\end{align}
Recall, that the $d_l$ are given by $d_{l}=\max_{|J|=l}d_J$.
\end{thm}
\begin{remark}\label{mtalc}
Theorem \ref{mta} also holds in the parametric case, i.e., if $f_j\in V_j$ for $j\in J$. In this case, only Assumption \ref{anglesep} has to be satisfied, Assumption \ref{hass} and the condition \eqref{eq:lowbound} disappear. Moreover, in \eqref{eq:mres} the term $\exp(-3n/(16 d_{q^*}))$ can be dropped. 
\end{remark}
The bound \eqref{eq:mres} yields the following simpler one
\begin{align}
\mathbb{P}\left( J_0\subseteq \hat{J}_0 \right)&\geq 1-\mathbb{P}\left(\mathcal{E}_{\delta,q^*}^c\right)-\exp\left(-\frac{3}{16}\frac{n}{d_{q^*}}\right)\nonumber\\
&-\left(\frac{eq}{q^*}\right)^{q^*}4\exp\left( -c_1\frac{n^2(1-\rho_{q^*}^2)^2\kappa^2}{\sigma^4d_{q^*} +\sigma^2n(1-\rho_{q^*}^2)\kappa}\right).\label{eq:qj2}
\end{align}
This can be seen as follows. First, we successively apply the bounds $\kappa_l\geq \kappa$ and $d_{q^*-s+l}\leq d_{q^*}$. Then, we use the following combinatorial result (for a proof see, e.g., \cite[Proposition 2.5]{M})
\begin{equation}\label{eq:comblem}
\sum_{j=0}^{q^*}\binom{q}{j}\leq\left(\frac{eq}{q^*}\right)^{q^*}.
\end{equation}
From \eqref{eq:qj2}, we conclude:
\begin{corollary}\label{mtacf} Suppose that the assumptions of Theorem \ref{mta} hold. Then for each constant $c_2>0$, there is a constant $c_3>0$ (depending only on $c_1$ and $c_2$) such that
\[\mathbb{P}\left( J_0\subseteq \hat{J}_0 \right)\nonumber\geq 1-\mathbb{P}\left(\mathcal{E}_{\delta,q^*}^c\right)-q^{-c_2},\]
provided that
\[\max\left\{\frac{\sigma^2\sqrt{q^*d_{q^*}\log(eq/q^*)}}{(1-\rho_{q^*}^2)\kappa},\frac{\sigma^2q^*\log(eq/q^*)}{(1-\rho_{q^*}^2)\kappa},d_{q^*}\log q\right\}\leq c_3n.\]
\end{corollary}
\begin{remark}
Corollary \ref{mtacf} also holds if the assumptions of Remark~\ref{mtalc} are satisfied. In this case, the term $d_{q^*}\log q$ can be dropped.
\end{remark}
Next, we present another analysis of \eqref{eq:mres} in the case that $q^*=s$.
Then $J_0\subseteq \hat{J}_0$ if and only if $J_0= \hat{J}_0$. Thus, we can rewrite \eqref{eq:mres} as
\begin{align}
&\mathbb{P}\left( J_0\neq \hat{J}_0 \right)\leq\mathbb{P}\left(\mathcal{E}_{\delta,q^*}^c\right)+\exp\left(-\frac{3}{16}\frac{n}{d_{q^*}}\right)\nonumber\\
&+\sum_{l=1}^{s}\sum_{m=0}^{l}\binom{s}{l}\binom{q-s}{m}
4\exp\left( -c_1\frac{n^2(1-\rho_{s}^2)^2\kappa_l^2}{\sigma^4d_{l} +\sigma^2n(1-\rho_{s}^2)\kappa_l}\right).\label{eq:qj}
\end{align}
Applying  $\kappa_l\geq (1-\epsilon_s)l\kappa_1$, $d_l\leq ld_1$, and 
\begin{equation*}
\sum_{m=0}^{l}\binom{s}{l}\binom{q-s}{m}\leq q^{2l},
\end{equation*}
the last expression in \eqref{eq:qj} can be bounded by
\[\sum_{l=1}^{s}
4q^{2l}\exp\left( -c_1\frac{l(n(1-\rho_{s}^2)(1-\epsilon_s)\kappa_1)^2}{\sigma^4d_{1} +\sigma^2n(1-\rho_{s}^2)(1-\epsilon_s)\kappa_1}\right).\]
We obtain:
\begin{corollary}\label{mtac} Suppose that the assumptions of Theorem \ref{mta} hold. Moreover, suppose that $q^*=s$. Then for each constant $c_2>0$, there is a constant $c_3>0$ (depending only on $c_1$ and $c_2$) such that
\begin{equation}\label{eq:sceq}
\mathbb{P}\left( J_0\neq \hat{J}_0 \right)\leq\mathbb{P}\left(\mathcal{E}_{\delta,q^*}^c\right)+q^{-c_2},
\end{equation}
provided that 
\begin{equation*}
\max\left\{\frac{\sigma^2\sqrt{d_1\log q}}{(1-\rho_s^2)(1-\epsilon_s)\kappa_1},\frac{\sigma^2\log q}{(1-\rho_s^2)(1-\epsilon_s)\kappa_1},d_s\log q\right\}\leq c_3n.
\end{equation*}
\end{corollary}
\begin{remark}
Corollary \ref{mtac} also holds if the assumptions of Remark \ref{mtalc} are satisfied. In this case, the term $d_{s}\log q$ can be dropped. In the special case that the covariates are also independent, we have $\rho_s=\epsilon_s=0$ and the conditions become
\begin{equation*}
\max\left\{\frac{\sigma^2\sqrt{d_1\log q}}{\kappa_1},\frac{\sigma^2\log q}{\kappa_1}\right\}\leq c_3n.
\end{equation*}
Note that these conditions are also necessary (see \cite{W4}).
\end{remark}


Finally, we mention that in the case $q^*=s$, the conditions in Corollary \ref{mtacf} and \ref{mtac} are both consequences of a more general condition. One can show that for each $c_2>0$, there is a $c_3>0$ such that \eqref{eq:sceq} holds, provided that
for $l=1,\dots,s$, 
\begin{equation}\label{eq:pc}
\max\left\{\frac{\sigma^2\sqrt{ld_{l}\log(eq/l)}}{(1-\rho_{s}^2)\kappa_l},\frac{\sigma^2l\log(eq/l)}{(1-\rho_{s}^2)\kappa_l},d_{s}\log q\right\}\leq c_3n.
\end{equation}
Note that Corollary \ref{mtacf} follows from the bounds $\kappa_l\geq \kappa$ and the fact that $l\log(eq/l)$ is increasing  in $l$ for $1\leq l\leq q$, and Corollary \ref{mtac} follows (up to the constant $e$ in the logarithm) from the bounds $\kappa_l\geq (1-\epsilon_s)l\kappa_1$ and $d_l\leq ld_1$ and the fact that $\log(eq/l)$ is decreasing in $l$.

\section{Structured random matrices and the event $\mathcal{E}_{\delta,q^*}$}\label{grip}
\subsection{Independent covariates and the RIP}\label{icrip} In this subsection, we suppose that $X_1,\dots,X_n$ are independent, which implies that the spaces $V_1,\dots,V_q$ are orthogonal in $L^2(\mathbb{P}^X)$. In this particular case, we rewrite the event $\mathcal{E}_{\delta,q^*}$ as a restricted (block)-isometry property. This allows us to apply known concentration inequalities.

For $j=1,\dots,q$, let $\{\phi_{jk}\}_{1\leq k\leq \dim V_j}$ be an orthonormal basis of $V_j$. Then we define the $n\times \dim V_j$-matrix
\[A_j=\frac{1}{\sqrt{n}}\left(\phi_{jk}(X_j^i)\right)_{1\leq i\leq n, 1\leq k\leq \dim V_j}\]
and for $J\subseteq\lbrace1,\dots,q\rbrace$, we define the $n\times d_J$-matrix
$A_J=\left(A_j\right)_{j\in J}$ (we abbreviate $A_{\{1,\dots,q\}}$ as $A$).
With these definitions, it is easy to see that $\mathcal{E}_{\delta,J}$ is the event such that
\[(1-\delta)\|z_J\|_2^2\leq \|A_Jz_J\|_2^2\leq (1+\delta)\|z_J\|_2^2\]
for all $z_J\in\mathbb{R}^{d_J}$. Here, we have used that the spaces $V_1,\dots,V_q$ are orthogonal. Thus, if we define
\[\delta_{q^*}=\max_{J\subseteq\lbrace1,\dots,q\rbrace, |J|\leq q^*}\|A_{J\cup J_0}^TA_{J\cup J_0}-I\|_{\operatorname{op}},\]
then we have \[\mathcal{E}_{\delta,q^*}=\left\{\delta_{q^*}\leq \delta\right\}.\]
The constant $\delta_{q^*}$ is bounded by the restricted isometry constant of order $d_{2q^*}$ of the matrix $A$ (see \cite[Definition 2.4]{R}). Note that the restricted isometry constant plays a prominent role in the theory of sparse recovery. Moreover, there exist many concentration inequalities for the restricted isometry constant in many ensembles of random matrices. We give two examples.
\begin{example}\label{lgm}
Consider the model $Y=\sum_{j=1}^qX_j\beta_j+\epsilon$, where the $X_j$ are independent centered Gaussian random variables and the $\beta_j$ are real numbers. Then $A$ is a Gaussian random matrix (the entries are independent Gaussian random variables, each with expectation zero and variance $1/n$), and \cite[Theorem 5.2]{BDDW} implies that there exist constants $c_3,c_4>0$ depending only on $\delta$ such that $\mathbb{P}(\delta_{q^*}\leq \delta)\geq 1-2\exp(-c_4n)$, provided that $q^*\log(q/q^*)\leq c_3n$. Combining this with Corollary \ref{mtac}, we obtain (in the case $q^*=s$) that $\mathbb{P}( J_0\neq \hat{J}_0 )\leq 2\exp(-c_4n)+q^{-c_2}$, provided that
\begin{equation*}
\max\left\{s\log(q/s),\frac{\sigma^2\log q}{\kappa_1}\right\}\leq c_3n.
\end{equation*}
These conditions are also known to be necessary (see, e.g., \cite[Section 2.6]{R} for the setting without noise and \cite[Theorem 2]{Wa} for the noisy setting).
\end{example} 
\begin{example}
Consider the nonparametric case where the $X_1,\dots,X_n$ are independent and uniformly distributed on $[0,1]$. Then the trigonometric bases of the $V_j$ are also orthonormal bases and we can apply \cite[Theorem 8.4]{R} (recall that the constant $\delta_{q^*}$ is bounded by the restricted isometry constant of order $d_{2q^*}$) which says that there are constants $c_3,c_4>0$ such that for $\delta\leq 1/2$,
$\mathbb{P}(\delta_{q^*}> \delta)\leq \exp(-c_4n\delta^2/d_{2q^*})$, provided that $d_{2q^*}\log^2(100d_{2q^*})\log(4d_q)\log(10n)\leq c_3n\delta^2$.
\end{example} 

\subsection{A general upper bound for $\mathbb{P}(\mathcal{E}_{\delta,q^*}^c)$} In this section, we give a general upper bound for the probability that the event $\mathcal{E}_{\delta,q^*}^c$ occurs. This upper bound is a generalization of \cite[Theorem 3.3]{RV2} and \cite[Theorem 8.1 and 8.4]{R}. The derivation will consist in two steps. The first step is the following generalization of Theorem 3.6 by Rudelson and Vershynin \cite{RV2}. 
\begin{proposition}\label{rudver} Let Assumptions \ref{anglesep} and \ref{hass} be satisfied. Then there is a universal constant $C_1>0$ such that 
\begin{equation}\label{eq:rudvereq}
\mathbb{E}\left[\sup_{g\in V_J,|J|\leq 2q^*,\|g\|\leq 1}\left|\|g\|_n^2-\| g\|^2\right|\right]\leq C_1\sqrt{\frac{d_{2q*}}{c(1-\epsilon_{2q^*})n}}\log^2(d_q\vee n),
\end{equation}
provided that the last expression is smaller than $1$.
\end{proposition}
A proof of Proposition \ref{rudver} is given in Appendix \ref{rudverpr}. The second step is an application of Talagrand's inequality (see \cite{Tal}). Here, we state a version of Talagrand's inequality presented in \cite[Corollary 2]{BM}:
\begin{thm}[Talagrand's inequality]\label{talineq} Consider $n$ independent and identically distributed
random variables $X^1,\dots,X^n$ taking values in some measurable space $(S,\mathcal{B})$. Let $\mathcal{G}$ be a countable family of real-valued measurable functions on $(S,\mathcal{B})$ that are uniformly bounded by some constant $b$. Let $Z=\sup_{g\in\mathcal{G}}\left|\frac{1}{n}\sum_{i=1}^ng(X^i)-\mathbb{E}\left[g(X^i)\right]\right|$ and $v=\sup_{g\in\mathcal{G}}\mathbb{E}\left[ g^2(X^1)\right]$. Then for every positive number $\lambda$, 
\begin{equation*}
\mathbb{P}\left( Z \geq 2\mathbb{E}\left[Z\right]+\lambda\right)\leq 3\exp\left(-n\kappa\left(\frac{\lambda^2}{v}\wedge\frac{\lambda}{b}\right)\right),
\end{equation*}
where $\kappa$ is a universal constant.
\end{thm}
We want to apply Talagrand's inequality to the family $\mathcal{G}=\{g^2:g\in V_J,|J|\leq 2q^*,\|g\|\leq 1\}$. This family is not countable, but the value of $Z$ does not change if we restrict the supremum to a countable and dense subset (note that the $V_J$ are finite-dimensional spaces). For $J\subseteq\lbrace1,\dots,q\rbrace$, let
\[
 \varphi_J=\frac{1}{\sqrt{d_J}}\sup_{0\neq g\in V_J}\frac{\Vert g\Vert_\infty}{\Vert g\Vert}.
\]
Moreover, let $\varphi_{2q^*}=\max_{|J|\leq 2q^*}\varphi_{J}$. Under Assumptions \ref{anglesep} and \ref{hass}, we have
\begin{equation}\label{eq:zbl}
\varphi_{2q^*}^2\leq \frac{2}{c(1-\epsilon_{2q^*})},
\end{equation}
the details are given in Appendix \ref{zb}. Therefore, for all $g^2\in \mathcal{G}$, we have
\[\|g\|_\infty^2\leq \varphi_{2q^*}^2d_{2q^*}\|g\|^2\leq\frac{2d_{2q^*}}{c(1-\epsilon_{2q^*})}.\]
Using this and $\mathbb{E}\left[g^4(X^1)\right]\leq \|g\|_\infty^2\|g\|^2$, we conclude that $b,v\leq 2d_{2q^*}/(c(1-\epsilon_{2q^*}))$.
Now, suppose that the last expression in \eqref{eq:rudvereq} is smaller than $\delta/4$, $0<\delta<1$. Then Theorem \ref{talineq}, applied with $\lambda=\delta/2$, yields
\begin{align*}
\mathbb{P}\left(\mathcal{E}_{\delta,q^*}^c\right)&\leq\mathbb{P}\left(\sup_{g\in V_J,|J|\leq 2q^*,\|g\|\leq 1}\left|\|g\|_n^2-\| g\|^2\right|> \delta\right)\\
&\leq 3\exp\left(-n\kappa\frac{c(1-\epsilon_{2q^*})\delta^2}{8d_{2q^*}}\right).
\end{align*}
We have shown:
\begin{thm}\label{rudvertal} Let Assumptions \ref{anglesep} and \ref{hass} be satisfied. Let $c_4=c\kappa/8$ and $c_3=\sqrt{c}/(4C_1)$, where $C_1$ and $\kappa$ are the constants in Proposition \ref{rudver} and Talagrand's inequality, respectively. Let $\delta\in(0,1)$. Suppose that 
\begin{equation*}
\sqrt{\frac{d_{2q^*}}{(1-\epsilon_{2q^*})n}}\log^2(d_q\vee n)\leq c_3\delta.
\end{equation*}
Then 
\[\mathbb{P}\left(\mathcal{E}_{\delta,q^*}^c\right)\leq 3\exp\left(-c_4\frac{(1-\epsilon_{2q^*})n\delta^2}{d_{2q^*}}\right).\]
\end{thm}
\subsection{Conditions for variable selection}\label{dsae} In this section, we combine Corollary \ref{mtac} with Theorem \ref{rudvertal}.
Therefore, suppose that the assumptions of Theorem \ref{mta} hold. To simplify the exposition, we will treat the quantities $\alpha=\min_j\alpha_j$, $K=\max_{j}K_j$, and $c$ from Assumption \ref{hass} and the geometric quantities $\rho_{s}$, $\epsilon_{2s}$, and $\epsilon'_{s}$ as constants. Moreover, we assume that $q^*=s$ and that $q\geq n$. Recall from \eqref{eq:lowbound} that in this case it suffices to choose the $m_j$ of size constant times $(s/\kappa)^{1/(2\alpha)}$.
By the inequalities $\kappa_l\geq l(1-\epsilon_l)\kappa_1$, we have that $\kappa$ is bounded from below by a constant times $\kappa_1$, which in turn implies that the $m_j$ can be chosen of size constant times $(s/\kappa_1)^{1/(2\alpha)}$.
Inserting this into Corollary \ref{mtac} and Theorem \ref{rudvertal} (let, e.g., $\delta=1/2$), we obtain:
\begin{corollary}\label{jflg} Make the above assumptions. Then for each constant $c_2>0$, there are constants $c_3>0$ and $c_5>0$ such that
\begin{equation*}
\mathbb{P}\left( J_0\neq \hat{J}_0 \right)\leq q^{-c_2}+q^{-c_5\log^3 q},
\end{equation*}
provided that 
\begin{equation}\label{eq:lfk}
\max\left\{\frac{\sigma^2s^{1/(4\alpha)}\sqrt{\log q}}{\kappa_1^{(4\alpha+1)/(4\alpha)}},\frac{\sigma^2\log q}{\kappa_1},\frac{s^{(2\alpha+1)/(2\alpha)}\log^4 q}{\kappa_1^{1/(2\alpha)}}\right\}\leq c_3n.
\end{equation}
\end{corollary}
\begin{remark} 
In \cite{W4}, it is shown that the condition
\begin{equation*}
\max\left\{\frac{\sigma^2\sqrt{\log q}}{\kappa_1^{(4\alpha+1)/(4\alpha)}},\frac{\sigma^2\log q}{\kappa_1}\right\}\leq c_3n
\end{equation*}
is optimal in an additive Gaussian white noise model. Obviously, this condition is weaker than \eqref{eq:lfk}. In \eqref{eq:lfk},  we have the additional factor $s^{1/(4\alpha)}$ in the first term, and we have an additional term coming from the event $\mathcal{E}_{\delta,q^*}$ (note that this event disappears in the Gaussian white noise framework). 
\end{remark}
\subsection{Estimation of single components}
The proposed selection criterion can be seen as a method to reduce the dimension of the model. We start with $n$ independent observations of a sparse additive model with $q$ covariates and an unknown subset $J_0$ of indices corresponding to the non-zero components, and we end up with a subset $\hat{J}_0$ such that $|\hat{J}_0|\leq q^*$ and $J_0\subseteq\hat{J}_0$ with high probability. More precisely, if $\{J_0\subseteq\hat{J}_0\}$ holds, then we have successfully reduced the model \eqref{eq:sam} to
\begin{equation}\label{eq:samr}
Y=\sum_{j\in\hat{J}_0}f_j(X_j)+\epsilon.
\end{equation}
We now consider the problem of estimating a single component $f_j$ of the model \eqref{eq:sam} with $j\in J_0$. We may assume without loss of generality that $j=1$. To simplify the exposition, we make the same assumptions as in the previous Section \ref{dsae}. 
We split the sample into two parts.  More precisely, we assume that we observe an even number of independent copies $(Y^1,X^1),\dots,(Y^{2n},X^{2n})$ of $(Y,X)$. The estimator $\hat{J}_0$ of $J_0$ is constructed as in Section \ref{selcrit} using the sample $(Y^1,X^1),\dots,(Y^{n},X^{n})$, and the estimator $\hat{f}_1$ of $f_1$ is constructed as in \cite[Section 2.3]{W} using $\hat{J}_0$ and the sample $(Y^{n+1},X^{n+1}),\dots,(Y^{2n},X^{2n})$. We have
\[
\mathbb{E}\left[\Vert f_1-\hat{f}_1^*\Vert^2\right]\leq \mathbb{E}\left[1_{\{\hat{J}_0= J_0\}}\Vert f_1-\hat{f}_1^*\Vert^2\right]+(\left\|f_1\right\|+k_n)^2\mathbb{P}\left(  J_0\neq \hat{J}_0 \right)
\]
meaning that we can apply \cite[Corollary 2]{W} to the first term (note that \cite[Assumption 1 and 2]{W} are a consequence of Assumption \ref{anglesep}) and Corollary \ref{jflg} to the second term. 
\begin{corollary} Make the same assumptions as in Corollary \ref{jflg}. Then there are constants $c_3>0$, $C>0$ such that
\[\mathbb{E}\left[\Vert f_1-\hat{f}_1^*\Vert^2\right]\leq
Cn^{\frac{-2\alpha_1}{2\alpha_1+1}},\]
provided that \eqref{eq:lfk} is satisfied and that
\begin{equation*}
s^{(2\alpha+1)/(2\alpha)}n^{\frac{2\alpha_1}{2\alpha(2\alpha_1+1)}}\log^4 n\leq c_3n.
\end{equation*}
\end{corollary} 
\section{Outline of the proof of Theorem \ref{mta}}\label{opt}

\subsection{The finite sample geometry} In this section, we present empirical versions of Assumption \ref{anglesep} and Proposition \ref{fundlem}. Throughout this section, let $0<\delta<1$ be the constant in Theorem \ref{mta}. 
Recall that in Section \ref{tmr}, we defined the events 
\begin{equation*}
\mathcal{E}_{\delta,J}=\left\lbrace (1-\delta)\Vert g_J\Vert^2\leq\Vert g_J\Vert_n^2\leq (1+\delta)\Vert g_J\Vert^2\ \text{ for all }g_J\in V_J\right\rbrace
\end{equation*}
for $J\subseteq\lbrace 1,\dots,q\rbrace$. 
Written in the equivalent form of Remark \ref{angeqsin}, we have:  
\begin{lemma}\label{lemaa} Let Assumption \ref{anglesep} be satisfied. Let $J_1,J_2\subseteq\lbrace1,\dots,q\rbrace$ be two subsets such that $J_1\cap J_2=\emptyset$ and $\vert J_1\vert, \vert J_2\vert\leq q^*$. If $\mathcal{E}_{\delta,J_1\cup J_2}$ holds, then we have
\begin{equation}\label{eq:empanglesep}
 \left\| g_{J_1}+g_{J_2}\right\|_n^2\geq  \frac {(1-\delta)}{(1+\delta)}(1-\rho_{q^*}^2)\left\| g_{J_1}\right\|_n^2
\end{equation} 
for all $g_{J_1}\in V_{J_1}$, $g_{J_2}\in V_{J_2}$.
\end{lemma}
\begin{proof}
Under the assumptions of Lemma \ref{lemaa}, we have
\begin{align*}
 \Vert g_{J_1}+g_{J_2}\Vert^2_n &\geq (1-\delta)\Vert g_{J_1}+g_{J_2}\Vert^2\\
   						&\geq (1-\delta)(1-\rho_{q^*}^2)\Vert g_{J_1}\Vert^2\\
   						&\geq \frac{(1-\delta)}{(1+\delta)}(1-\rho_{q^*}^2)\Vert g_{J_1}\Vert^2_n.
\end{align*}
This completes the proof.
\end{proof}
Applying \eqref{eq:empanglesep} as in the proof of Proposition \ref{fundlem}, we obtain:
\begin{proposition}\label{fundlemem}
Let Assumption \ref{anglesep} be satisfied.
Let $J\subseteq\lbrace1,\dots,q\rbrace$ be a subset such that $\vert J\vert\leq q^*$ and $J_0\setminus J\neq\emptyset$. Let $v=\sum_{j\in J_0}v_j$ with $v_j\in V_j$ for $j\in J_0$. If $\mathcal{E}_{\delta,J\cup J_0}$ holds, then we have
\begin{equation*}
\Vert\hat{\Pi}_{J_0}v\Vert_n^2-\Vert \hat{\Pi}_{J}v\Vert_n^2=\Vert v-\hat{\Pi}_{J}v\Vert_n^2\geq\frac{(1-\delta)}{(1+\delta)}(1-\rho_{q^*}^2)\Big\Vert \sum_{j\in J_0\setminus J}v_j\Big\Vert_n^2.
\end{equation*} 
\end{proposition}
By decomposing $f$ as $v+f-v$  with $v=\sum_{j\in J_0}\Pi_{V_j}f_j$, we can apply Proposition \ref{fundlemem} to $v$ and Lemma \ref{testdev2} to $f-v$. The result is the following empirical version of Proposition \ref{fundlem}.
\begin{proposition}\label{fln}
Let Assumption \ref{anglesep} and Assumption \ref{hass} be satisfied. Suppose that \eqref{eq:lowbound} is satisfied for $j=1,\dots,q$. Let $J\subseteq\lbrace1,\dots,q\rbrace$ be a subset such that $\vert J\vert\leq q^*$ and $J_0\setminus J\neq\emptyset$. Let $l=|J_0\setminus J|$. If $\mathcal{E}_{\delta,J\cup J_0}\cap\mathcal{A}$ holds, then we have
\begin{equation}\label{eq:thrdvers}
\Vert\hat{\Pi}_{J_0}f\Vert_n^2-\Vert \hat{\Pi}_{J}f\Vert_n^2
\geq\frac{1}{2}\frac{(1-\delta)^2}{(1+\delta)}(1-\rho_{q*}^2)\kappa_l, 
\end{equation} 
provided that
\begin{equation}\label{eq:ccond}
(2/3)(1-\sqrt{c'})^2-8\frac{(1+\delta)}{(1-\delta)^2}c'\geq 1/2.
\end{equation}
\end{proposition}
A proof of Proposition \ref{fln} is given in Appendix \ref{appa}. In the absence of noise, Proposition \ref{fundlemem} and \ref{fln} already prove Theorem \ref{mta}. In fact, if the event $\mathcal{E}_{\delta,q^*}\cap\mathcal{A}$ holds, then \eqref{eq:sc} selects a subset $\hat{J}_0\subseteq\{1,\dots,q\}$ with $|\hat{J}_0|\leq q^*$ and $J_0\subseteq \hat{J}_0$. Proposition \ref{fln} applies to the nonparametric setting, while Proposition \ref{fundlemem} applies if the components $f_j$ satisfy $f_j\in V_j$, the latter being a commonly used setting in the theory of compressive sensing (see, e.g., \cite{ra10} and the references therein). 
\subsection{End of the proof of Theorem \ref{mta}} We have
\begin{multline*}
\mathbb{P}\left(  J_0\nsubseteq \hat{J}_0 \right)=\mathbb{P}\left(  J_0\setminus\hat{J}_0 \neq\emptyset\right)\\
\leq \mathbb{P}\left(\exists J\subseteq\lbrace 1,\dots,q\rbrace, |J|\leq q^*\text{ with }J_0\setminus J \neq\emptyset\right.\\\left.\text{ and }\Vert\hat{\Pi}_J\mathbf{Y}\Vert^2_n-d_J/n\geq\Vert\hat{\Pi}_{J_0}\mathbf{Y}\Vert^2_n-d_{J_0}/n\right).
\end{multline*}
Applying the union bound, we obtain
\begin{multline*}
\mathbb{P}\left(  J_0\nsubseteq \hat{J}_0 \right)\leq\mathbb{P}\left(\mathcal{E}_{\delta,q^*}^c\right)+\mathbb{P}\left(\mathcal{A}^c\right)\\+\sum_{\substack{ J\subseteq\lbrace 1,\dots,q\rbrace\\ |J|\leq q^*,J_0\setminus J \neq\emptyset}}\mathbb{P}\left( \mathcal{E}_{\delta,J\cup J_0}\cap\mathcal{A}\cap\Vert\hat{\Pi}_J\mathbf{Y}\Vert^2_n-d_J/n\geq\Vert\hat{\Pi}_{J_0}\mathbf{Y}\Vert^2_n-d_{J_0}/n\right),
\end{multline*}
where $\mathcal{A}$ is the event defined in Proposition \ref{fln}. We have:
\begin{lemma}\label{testdev} Let Assumptions \ref{anglesep} and \ref{hass} be satisfied. Suppose that \eqref{eq:lowbound} is satisfied for $j=1,\dots,q$. Let $J\subseteq\lbrace1,\dots,q\rbrace$ be a subset such that $\vert J\vert\leq q^*$ and $J_0\setminus J\neq\emptyset$. Let $l=|J_0\setminus J|$. Then there is a constant $c_1$ depending only on $\delta$ (given explicitly in the proof) such that
\begin{align*}
\mathbb{P}&\left(\mathcal{E}_{\delta,J\cup J_0}\cap\mathcal{A}\cap \Vert\hat{\Pi}_J\mathbf{Y}\Vert^2_n-\sigma^2d_J/n\geq\Vert\hat{\Pi}_{J_0}\mathbf{Y}\Vert^2_n-\sigma^2d_{J_0}/n\right)\\
&\leq 4\exp\left( -c_1\frac{n^2(1-\rho_{q^*}^2)^2\kappa_l^2}{\sigma^4d_{q^*-s+l} +\sigma^2n(1-\rho_{q^*}^2)\kappa_l}\right).
\end{align*}
\end{lemma}
A proof of Lemma \ref{testdev} is given in Appendix \ref{adev}. Thus
\begin{align*}
&\mathbb{P}\left(  J_0\nsubseteq \hat{J}_0 \right)\leq \mathbb{P}\left(\mathcal{E}_{\delta,q^*}^c\right)+\mathbb{P}\left(\mathcal{A}^c\right)\\
&+\sum_{l=1}^{s}\sum_{m=0}^{q^*-(s-l)}\binom{s}{l}\binom{q-s}{m}
4\exp\left( -c_1\frac{n^2(1-\rho_{q^*}^2)^2\kappa_l^2}{\sigma^4d_{q^*-s+l} +\sigma^2n(1-\rho_{q^*}^2)\kappa_l}\right).
\end{align*}
Now apply Lemma \ref{testdev2}. This completes the proof.\qed

\appendix
\section{Proof of Remark \ref{angeqsin}}\label{efa}
Suppose that \eqref{eq:anglesepeq} holds, and let $h_{J_1}\in H_{J_1}$ and $h_{J_2}\in H_{J_2}$. Then $\Vert h_{J_1}+h_{J_2}\Vert^2\geq \Vert h_{J_1}\Vert^2-2\rho_{q^*}\Vert h_{J_1}\Vert\Vert h_{J_2}\Vert+\Vert h_{J_2}\Vert^2$ and \eqref{eq:mmm} follows from the inequality $2\rho_{q^*}\Vert h_{J_1}\Vert\Vert h_{J_2}\Vert\leq \rho_{q^*}^2\Vert h_{J_1}\Vert^2+\Vert h_{J_2}\Vert^2$.

Conversely, suppose that \eqref{eq:mmm} holds, and let $h_{J_1}\in H_{J_1}$ and $h_{J_2}\in H_{J_2}$. We may assume without loss of generality that $h_{J_2}\neq 0$ and that $\Vert h_{J_2}\Vert=1$. Then $\left\|h_{J_1}\right\|^2-\langle h_{J_1},h_{J_2}\rangle^2=\left\|h_{J_1}-\langle h_{J_1},h_{J_2}\rangle h_{J_2}\right\|^2\geq(1-\rho_{q^*}^2)\left\|h_{J_1}\right\|^2$
which gives \eqref{eq:anglesepeq}. This completes the proof.\qed

\section{Proof of Lemma \ref{ric}}\label{eas}
Let $J_1,J_2\subseteq\lbrace1,\dots,q\rbrace$ be two subsets  satisfying $J_1\cap J_2=\emptyset$ and $\vert J_1\vert, \vert J_2\vert\leq q^*$. Applying \eqref{eq:riceq} and \eqref{eq:ric2}, we see that
\[
\left\| f_{J_1}+f_{J_2}\right\|^2\geq \frac{1-\epsilon_{2q^*}}{1+\epsilon_{q^*}'}\left(\left\| f_{J_1}\right\|^2+\left\|f_{J_2}\right\|^2\right)
\]
for all $f_{J_1}\in H_{J_1}$, $f_{J_2}\in H_{J_2}$. Thus Remark \ref{angeqsin} gives the ``if'' part. Conversely, applying \eqref{eq:anglesepeq} iteratively, one gets for instance \[1-\epsilon_{2q^*}\geq (1-\rho_{q^*}^2)^{\log_2q^*+1}\]
which gives the ``only if'' part.\qed

\section{Proof of Lemma \ref{lemappr}}\label{plemappr}
Let $\sum_{k=1}^\infty \theta_k\phi_k\in \tilde{W}_j(\alpha_j,K_j)$. Then there is a constant $c_{\alpha_j}$ depending only on $\alpha_j$ such that (see, e.g., \cite[Proof of Lemma 1.8 and Theorem 1.9]{T})
\begin{equation}\label{eq:kg1}
\sum_{k>m_j}\theta_j^2\leq c_{\alpha_j}K_j^2m_j^{-2\alpha_j}
\end{equation}
and 
\begin{equation}\label{eq:kg2}
\left(\sum_{k>m_j}|\theta_j|\right)^2\leq c_{\alpha_j}K_j^2m_j^{1-2\alpha_j}.
\end{equation}
We now define $U_j=V_j+\mathbb{R}$, where $\mathbb{R}$ denotes the constant functions. Since $V_j$ and $\mathbb{R}$ are orthogonal for $j=1,\dots,q-1$, and since $V_q=U_q$, we have $\Pi_{V_j}h_j=\Pi_{U_j}h_j$ for $h_j\in H_j$ and $j=1,\dots,q$. 
Now, let $f_j\in H_j\cap\tilde{W}_j(\alpha_j,K_j)$. Then $f_j(x_j)=\sum_{k=1}^\infty \theta_k\phi_k(x_j)$. Let $q_j(x_j)=\sum_{j=1}^{m_j}\theta_k\phi_k(x_j)$. Then $q_j-\mathbb{E}[q_j(X_j)]\in V_j$ and thus $q_j\in U_j$. We conclude that
\begin{equation*}
\|f_j-\Pi_{V_j}f_j\|^2= \|f_j-\Pi_{U_j}f_j\|^2
\leq \|f_j-q_j\|^2\leq (1/c)\sum_{k>m_j}\theta_j^2	\label{eq:kg3}																	
\end{equation*}																		
and similarly that
\begin{align}
\|f_j-\Pi_{V_j}f_j\|_\infty^2&\leq 2\|f_j-q_j\|_\infty^2+2\|q_j-\Pi_{V_j}f_j\|_\infty^2\nonumber\\
																		&\leq 4\left(\sum_{k>m_j}|\theta_j|\right)^2+2(1/c) m_j\|q_j-\Pi_{U_j}f_j\|^2\nonumber\\
																		&\leq 4\left(\sum_{k>m_j}|\theta_j|\right)^2+2(1/c) m_j\|q_j-f_j\|^2\nonumber\\
																		&\leq 4\left(\sum_{k>m_j}|\theta_j|\right)^2+2(1/c)^2m_j\sum_{k>m_j}\theta_j^2\label{eq:kg4}
																		\end{align}	
Using \eqref{eq:kg1}-\eqref{eq:kg4}, we obtain Lemma~\ref{lemappr}.
This completes the proof.\qed

\section{Proof of Lemma \ref{testdev2}}\label{adev2} 
Let $v=\sum_{j\in J_0}v_j$ with $v_j=\Pi_{V_j}f_j$ for $j\in J_0$.
By \eqref{eq:af2}, we have \[\|f-v\|^2\leq c'(1-\rho_{q^*}^2)\kappa.\] Moreover, by Lemma \ref{lemappr} and the Cauchy-Schwarz inequality, we also have
\begin{equation}\label{eq:infa}\|f-v\|^2_\infty\leq q^*\sum_{j\in J_0}C_jK_j^2m_j^{1-2\alpha_j}\leq\frac{2d_{q^*}c(1-\rho_{q^*}^2)\kappa}{(1+\epsilon'_{q^*})}.\end{equation}
Thus, letting $x=c'(1-\rho_{q^*}^2)\kappa$, Bennett's inequality (see, e.g., \cite[Comment after Proposition 2.8]{M}) yields
\begin{align*}
\mathbb{P}\left(\|f-v\|^2_n> 2x\right)&\leq \mathbb{P}\left(\|f-v\|^2_n-\|f-v\|^2> x\right)\\
&\leq \exp\left(-\frac{nx^2}{2\|(f-v)^2\|^2+(2/3)\|f-v\|^2_\infty x}\right)\\
&\leq \exp\left(-\frac{3nx}{8\|f-v\|^2_\infty }\right).
\end{align*}
Using this and \eqref{eq:infa}, we obtain \eqref{eq:mn}. This completes the proof \qed

\section{Proof of Proposition \ref{rudver}}\label{rudverpr}
The proof is taken from \cite[proof of Theorem 3.6]{RV2} (see also \cite[proof of Theorem 8.1]{R}).  However, we have to modify several details. For $j=1,\dots, q-1$, the spaces $V_j$ are spanned by the functions $\psi_{jk}(x_j)=\phi_k(x_j)-\mathbb{E}[\phi_k(X_j)]$, $2\leq k\leq m_j$, and the space $V_q$ is spanned by $\psi_{qk}(x_q)=\phi_k(x_q)$, $1\leq k\leq m_q$. Thus each function in $\sum_{j=1}^qV_j$, can be written as 
$g_\alpha=\sum_{j,k}\alpha_{jk}\psi_{jk}$,
for some $\alpha=(\alpha_1^T,\dots,\alpha_q^T)^T\in\mathbb{R}^{d_q}$. Letting 
 \[T=\left\{\alpha\in\mathbb{R}^{d_q}:g_\alpha\in V_J,|J|\leq 2q^*,\|g_\alpha\|\leq 1\right\},\]
we have to show that there is a constant $C_1>0$ such that
\[E:=\mathbb{E}\left[\sup_{\alpha \in T}\left|\|g_\alpha\|_n^2-\|g_\alpha\|^2\right|\right]\leq C_1\sqrt{\frac{d_{2q*}}{c(1-\epsilon_{2q^*})n}}\log^2(d_q\vee n),\]
provided that the last expression is smaller than $1$.
Using the symmetrization lemma (see, e.g., \cite[Lemma 2.3.1]{VW}), we obtain
\[E\leq 2\mathbb{E}\left[\sup_{\alpha \in T}\frac{1}{n}\sum_{i=1}^n\delta^ig_{\alpha}^2(X^i)\right],\]
where $\delta^1,\dots,\delta^n$ are independent Rademacher random variables. Applying \cite[Corollary 2.2.8]{VW}, we have for a universal constant $C_2$,
\[E_1:=\mathbb{E}\left[\sup_{\alpha \in T}\frac{1}{n}\sum_{i=1}^n\delta^ig_{\alpha}^2(X^i)\bigg|X^1,\dots,X^n\right]\leq C_2\int_0^\infty\sqrt{\log N(T,d,u)}du,\]
where $N(T,d,u)$ denotes the minimal number of balls of radius $u$ in the semimetric $d$ needed to cover $T$ and $d$ is the given by 
\[d(\alpha,\beta)=\left(\frac{1}{n^2}\sum_{i=1}^n\left(g_\alpha^2(X^i)-g_\beta^2(X^i)\right)^2\right)^{1/2}.\]
Now,
\begin{align*}
d(\alpha,\beta)&\leq \left(\frac{1}{n^2}\sum_{i=1}^n\left(g_\alpha(X^i)+g_\beta(X^i)\right)^2\right)^{1/2}\max_{i=1,\dots,n}\left|g_\alpha(X^i)-g_\beta(X^i)\right|\\
&\leq \frac{2}{\sqrt{n}}\sup_{\alpha\in T}\|g_\alpha\|_n\max_{i=1,\dots,n}\left|g_\alpha(X^i)-g_\beta(X^i)\right|.
\end{align*}
Applying a linear change of variables, we obtain
\[E_1\leq \sup_{\alpha\in T}\|g_\alpha\|_n\frac{2C_2}{\sqrt{n}}\int_0^\infty\sqrt{\log N(T,\|\cdot\|_X,u)}du,\]
where the seminorm $\|\cdot\|_X$ is given by
\[\|\alpha\|_X=\max_{i=1,\dots,n}\left|g_\alpha(X^i)\right|=\max_{i=1,\dots,n}\left|\left\langle \alpha,x_i\right\rangle\right|.\] 
Here, the $x_i$ are the vectors of the basis functions evaluated at $X^i$. Note that the $x_i$ are uniformly bounded by $K=2\sqrt{2}$ and that the last expression coincides with the definition of $\|\cdot\|_X$ in \cite{RV2}.
Now, if $\alpha\in T$, then
\[\|\alpha\|_0=|\{j:\alpha_j\neq 0\}|\leq d_{2q^*}\]
and
\[\|\alpha\|_2\leq\frac{1}{\sqrt{c(1-\epsilon_{2q^*})}}.\]
The first inequality follows from the definition, the second one from Assumption \ref{hass}, \eqref{eq:riceq}, and $\|g_\alpha\|\leq 1$. Thus 
\[T\subseteq \frac{1}{\sqrt{c(1-\epsilon_{2q^*})}}D_2^{d_{2q^*},d_q},\]
where
\[D_2^{d_{2q^*},d_q}=\left\{\alpha\in\mathbb{R}^{d_q}:\|\alpha\|_0\leq d_{2q^*},\|\alpha\|_2\leq 1\right\}.\] 
Applying again a linear change of variables, we obtain
\begin{align*}
&E_1\leq\\& \sup_{\alpha\in T}\|g_\alpha\|_nC_2\sqrt{\frac{d_{2q^*}}{c(1-\epsilon_{2q^*})n}}\int_0^\infty\log^{1/2} N\left(\frac{1}{\sqrt{d_{2q^*}}}D_2^{d_{2q^*},d_q},\|\cdot\|_X,u\right)du.
\end{align*}
The above integral is the same as in \cite[(3.7)]{RV2} and can be bounded by $C_3\log(d_{2q^*})\sqrt{\log n}\sqrt{\log d_q}\leq C_3\log^2(d_q\vee n)$ (here, we use that the $x_i$ are uniformly bounded by $2\sqrt{2}$). We conclude that
\[E_1\leq C(q^*,q,n)\sup_{\alpha\in T}\|g_\alpha\|_n,\]
where
\[C(q^*,q,n)=C_2C_3\sqrt{\frac{d_{2q^*}}{c(1-\epsilon_{2q^*})n}}\log^2(d_q\vee n)\]
Using this and the Cauchy-Schwarz inequality, we obtain
\begin{align*}E&\leq C(q^*,q,n)\left(\mathbb{E}\left[\sup_{\alpha\in T}\|g_\alpha\|_n^2\right]\right)^{1/2}\\
&\leq C(q^*,q,n)\left(E+1\right)^{1/2}.
\end{align*}
If 
\[C(q^*,q,n)\leq 1,\]
then we get
\[E\leq 2C(q^*,q,n).\]
This completes the proof.\qed

\section{Proof of Equation \eqref{eq:zbl}}\label{zb} 
Let $g=\sum_{k=1}^m\theta_k\phi_k$. By the Cauchy-Schwarz inequality, we have  
\[
\left\|g\right\|_\infty^2\leq m\sum_{k=1}^m\theta_k^2=m\int_0^1g^2(x)dx.
\]
This implies that
\begin{equation}\label{eq:lll1}
\left\|g_j\right\|_\infty^2\leq(2/c)\dim V_j\left\|g_j\right\|^2
\end{equation}
for all $g_j\in V_j$.
Now, let $J\subseteq\lbrace1,\dots,q\rbrace$ be a subset with $|J|\leq 2q^*$. Applying \eqref{eq:lll1}, the Cauchy-Schwarz inequality, and Lemma \ref{ric}, we obtain
\begin{align*}
\left\| g_{J}\right\|_\infty\leq\sum_{j\in {J}}\left\|g_j\right\|_\infty
&\leq \sqrt{2/c}\sqrt{\sum_{j\in {J}}\dim V_j}\sqrt{\sum_{j\in {J}}\left\|g_j\right\|^2}\\
&\leq \sqrt{\frac{2}{c(1-\epsilon_{2q^*})}}\sqrt{\dim V_J}\left\|g_{J}\right\|
\end{align*}
for all $g_{J}=\sum_{j\in {J}}g_j\in V_{J}$. This completes the proof.\qed	

\section{Proof of Proposition \ref{fln}}\label{appa}
Let $v=\sum_{j\in J_0}v_j$ with $v_j=\Pi_{V_j}f_j$ for $j\in J_0$. We have
\begin{align}
&\Vert\hat{\Pi}_{J_0}f\Vert_n^2-\Vert \hat{\Pi}_{J}f\Vert_n^2\nonumber\\
&= \Vert\hat{\Pi}_{J_0}v\Vert_n^2+2\langle \hat{\Pi}_{J_0}v,\hat{\Pi}_{J_0}(f-v) \rangle_n+\Vert\hat{\Pi}_{J_0}(f-v)\Vert_n^2\nonumber\\
&-\Vert\hat{\Pi}_{J}v\Vert_n^2-2\langle \hat{\Pi}_{J}v,\hat{\Pi}_{J}(f-v) \rangle_n-\Vert\hat{\Pi}_{J}(f-v)\Vert_n^2\nonumber\\
&\geq \Vert\hat{\Pi}_{J_0}v\Vert_n^2-\Vert\hat{\Pi}_{J}v\Vert_n^2+2\langle \hat{\Pi}_{J_0}v-\hat{\Pi}_{J}v,f-v \rangle_n-\|f-v\|^2_n,\label{eq:fg1}
\end{align}
where the inequality holds since orthogonal projections are self-adjoint and lower the norm.
Since $\hat{\Pi}_{J_0}v=v$ and $\|v-\hat{\Pi}_{J}v\|_n^2=\|v\|_n^2-\|\hat{\Pi}_{J}v\|_n^2$, we get
\begin{align} 
&2\langle \hat{\Pi}_{J_0}v-\hat{\Pi}_{J}v,f-v \rangle_n\nonumber\\
&\leq (1/3)\|\hat{\Pi}_{J_0}v-\hat{\Pi}_{J}v\|_n^2+3\|f-v\|_n^2\nonumber\\
&= (1/3)\left(\Vert\hat{\Pi}_{J_0}v\Vert_n^2-\Vert\hat{\Pi}_{J}v\Vert_n^2\right)+3\|f-v\|_n^2,\label{eq:fg2}
\end{align}
where we also applied the bound $2xy\leq 3x^2+(1/3)y^2$. Combining \eqref{eq:fg1} and \eqref{eq:fg2}, we conclude that
\begin{equation}\label{eq:pyt}
\Vert\hat{\Pi}_{J_0}f\Vert_n^2-\Vert \hat{\Pi}_{J}f\Vert_n^2
\geq (2/3)\left(\Vert\hat{\Pi}_{J_0}v\Vert_n^2-\Vert\hat{\Pi}_{J}v\Vert_n^2\right)-4\|f-v\|_n^2.
\end{equation}
If $\mathcal{E}_{\delta,J\cup J_0}$ holds, then Proposition \ref{fundlemem} says that
\begin{equation*}
\big\Vert\hat{\Pi}_{J_0}v\big\Vert_n^2-\big\Vert \hat{\Pi}_{J}v\big\Vert_n^2
\geq\frac{(1-\delta)}{(1+\delta)}(1-\rho_{q^*}^2)\Vert  v_{J_0\setminus J}\Vert_n^2,
\end{equation*}
where $v_{J_0\setminus J}=\sum_{j\in J_0\setminus J}v_j$. If $\mathcal{E}_{\delta,J_0}$ holds, then 
\[
\Vert v_{J_0\setminus J}\Vert_n^2\geq (1-\delta)\Vert v_{J_0\setminus J}\Vert^2\geq (1-\delta)\left(\Vert f_{J_0\setminus J}\Vert-\Vert f_{J_0\setminus J}-v_{J_0\setminus J}\Vert\right)^2,
\]
where $f_{J_0\setminus J}=\sum_{j\in J_0\setminus J}f_j$.
As in \eqref{eq:af2}, we have 
\[\Vert f_{J_0\setminus J}-v_{J_0\setminus J}\Vert^2\leq c'(1-\rho_{q*}^2)\kappa\leq c'\kappa_l.\] 
Thus
\begin{equation*}
\Vert v_{J_0\setminus J}\Vert_n^2\geq (1-\delta)(1-\sqrt{c'})^2\kappa_l\label{eq:com}
\end{equation*}
If $\mathcal{E}_{\delta,J\cup J_0}$ holds, then we obtain
\begin{equation}\label{eq:pyt2}
\big\Vert\hat{\Pi}_{J_0}v\big\Vert_n^2-\big\Vert \hat{\Pi}_{J}v\big\Vert_n^2
\geq(1-\sqrt{c'})^2\frac{(1-\delta)^2}{(1+\delta)}(1-\rho_{q^*}^2)\kappa_l.
\end{equation}
If $\mathcal{E}_{\delta,J\cup J_0}\cap\mathcal{A}$ holds, then we conclude from \eqref{eq:pyt} and \eqref{eq:pyt2} that 
\begin{equation*}
\Vert\hat{\Pi}_{J_0}f\Vert_n^2-\Vert \hat{\Pi}_{J}f\Vert_n^2
\geq\frac{1}{2}\frac{(1-\delta)^2}{(1+\delta)}(1-\rho_{q*}^2)\kappa_l, 
\end{equation*} 
provided that \eqref{eq:ccond} is satisfied. This completes the proof.\qed

\section{Proof of Lemma \ref{testdev}}\label{adev}
We have
\[\Vert\hat{\Pi}_J\mathbf{Y}\Vert^2_n-\sigma^2d_J/n\geq\Vert\hat{\Pi}_{J_0}\mathbf{Y}\Vert^2_n-\sigma^2d_{J_0}/n\]
if and only if
\begin{align*}
&\Vert\hat{\Pi}_J\boldsymbol{\epsilon}\Vert^2_n-\Vert\hat{\Pi}_{J_0}\boldsymbol{\epsilon}\Vert^2_n-\sigma^2d_J/n+\sigma^2d_{J_0}/n+2\langle (\hat{\Pi}_J-\hat{\Pi}_{J_0})f,\boldsymbol{\epsilon}\rangle_n\\&\geq\Vert\hat{\Pi}_{J_0}f\Vert^2_n-\Vert\hat{\Pi}_Jf\Vert^2_n.
\end{align*}
If $\mathcal{E}_{\delta,J\cup J_0}\cap \mathcal{A}$ holds, then \eqref{eq:pyt}, \eqref{eq:ccond}, and \eqref{eq:thrdvers} yield
\begin{equation*}
\Vert\hat{\Pi}_{J_0}f\Vert_n^2-\Vert \hat{\Pi}_{J}f\Vert_n^2
\geq\frac{1}{2}\frac{(1-\delta)^2}{(1+\delta)}(1-\rho_{q^*}^2)\kappa_l
\end{equation*} 
and also
\begin{equation*}
\Vert\hat{\Pi}_{J_0}f\Vert_n^2-\Vert \hat{\Pi}_{J}f\Vert_n^2
\geq\frac{1}{2(1-\sqrt{c'})^2}\Vert v-\hat{\Pi}_{J}v\Vert_n^2\geq\frac{1}{2}\Vert v-\hat{\Pi}_{J}v\Vert_n^2.
\end{equation*}
Recall that the random variables $\epsilon^1,\dots,\epsilon^n$ are independent and Gaussian, each with expectation $0$ and variance $\sigma^2$. Moreover, they are independent of $X^1,\dots,X^n$. One can show that, conditioned on $X^1,\dots,X^n$ and if $\mathcal{E}_{\delta,J\cup J_0}$ holds, we have
\begin{equation*}
\Vert\hat{\Pi}_J\boldsymbol{\epsilon}\Vert^2_n-\Vert\hat{\Pi}_{J_0}\boldsymbol{\epsilon}\Vert^2_n\stackrel{d}{=}(\sigma^2/n)\chi^2(d_{J\setminus J_0})-(\sigma^2/n)\chi^2(d_{J_0\setminus J}),
\end{equation*}
where $\stackrel{d}{=}$ denotes equality in distribution, and
where $\chi^2(d_{J\setminus J_0})$ and $\chi^2(d_{J_0\setminus J})$ are chi-square distributed random variables with $d_{J\setminus J_0}$ and $d_{J_0\setminus J}$ degrees of freedom, respectively. 
Applying all these arguments and the union bound, we conclude that
\begin{align*}
\mathbb{P}&\left(\mathcal{E}_{\delta,J\cup J_0}\cap\mathcal{A}\cap \Vert\hat{\Pi}_J\mathbf{Y}\Vert^2_n-\sigma^2d_J/n\geq\Vert\hat{\Pi}_{J_0}\mathbf{Y}\Vert^2_n-\sigma^2d_{J_0}/n\right)\\
&\leq\mathbb{P}\left( \frac{\sigma^2}{n}\left(\chi^2(d_{J\setminus J_0})-d_{J\setminus J_0}\right)\geq\frac{1}{8}\frac{(1-\delta)^2}{(1+\delta)}(1-\rho_{q^*}^2)\kappa_l\right)\\
&+\mathbb{P}\left(\frac{\sigma^2}{n}\left(\chi^2(d_{J_0\setminus J})-d_{J_0\setminus J}\right)\leq -\frac{1}{8}\frac{(1-\delta)^2}{(1+\delta)}(1-\rho_{q^*}^2)\kappa_l\right)\\
&+\mathbb{P}\left(\mathcal{E}_{\delta,J\cup J_0}\cap\mathcal{A}\cap 2\langle (\hat{\Pi}_J-\hat{\Pi}_{J_0})f,\boldsymbol{\epsilon}\rangle_n\geq\frac{1}{4}\Vert v-\hat{\Pi}_{J}v\Vert_n^2\right).
\end{align*}
The first and the second term can be bounded by standard concentration inequalities for chi-square distributions.
\begin{lemma}\label{gh} Let $d$ be a positive integer. Then, for all $x\geq 0$, we have
\begin{equation*}
\mathbb{P}\left( \chi^2(d)-d\geq x\right)\leq \exp\left(-\frac{x^2}{2(2d+2x)} \right)
\end{equation*}
and
\begin{equation*}
\mathbb{P}\left( \chi^2(d)-d\leq -x\right)\leq \exp\left(-\frac{x^2}{4d} \right).
\end{equation*}
\end{lemma}
For a proof of this lemma see \cite[Lemma 1]{LM} and \cite[Lemma 8]{BM}. 
Since $|J_0\setminus J|=l$ and $|J\setminus J_0|\leq q^*-s+l$, we have $d_{J_0\setminus J},d_{J\setminus J_0}\leq d_{q^*-s+l}$. Applying this and Lemma \ref{gh}, we obtain
\begin{align}
&\mathbb{P}\left( \frac{\sigma^2}{n}\left(\chi^2(d_{J\setminus J_0})-d_{J\setminus J_0}\right)\geq\frac{1}{8}\frac{(1-\delta)^2}{(1+\delta)}(1-\rho_{q^*}^2)\kappa_l\right)\nonumber\\
&+\mathbb{P}\left(\frac{\sigma^2}{n}\left(\chi^2(d_{J_0\setminus J})-d_{J_0\setminus J}\right)\leq -\frac{1}{8}\frac{(1-\delta)^2}{(1+\delta)}(1-\rho_{q^*}^2)\kappa_l\right)\nonumber\\
&\leq  2\exp\left( -\frac{1}{32}\frac{c_\delta^2n^2(1-\rho_{q^*}^2)^2\kappa_l^2}{8\sigma^4d_{q^*-s+l} +c_\delta\sigma^2n(1-\rho_{q^*}^2)\kappa_l}\right),\label{eq:chisqev}
\end{align}
where $c_\delta=(1-\delta)^2/(1+\delta)$.
Thus it remains the third term.
It can be bounded by
\begin{align}
&\mathbb{P}\left(\mathcal{E}_{\delta,J\cup J_0}\cap \langle \hat{\Pi}_{J}v-v,\boldsymbol{\epsilon}\rangle_n\geq\frac{1}{16}\Vert v-\hat{\Pi}_{J}v\Vert_n^2\right)\nonumber \\
+&\mathbb{P}\left(\mathcal{E}_{\delta,J\cup J_0}\cap\mathcal{A}\cap \langle (\hat{\Pi}_{J}-\hat{\Pi}_{J_0})(f-v),\boldsymbol{\epsilon}\rangle_n\geq\frac{1}{16}\Vert v-\hat{\Pi}_{J}v\Vert_n^2\right)\label{eq:secte}.
\end{align}
These terms can be bounded by standard concentration inequalities for Gaussian random variables.
Applying \eqref{eq:pyt2}, we obtain
\begin{align*}
\mathbb{P}&\left(\mathcal{E}_{\delta,J\cup J_0}\cap \langle \hat{\Pi}_{J}v-v,\boldsymbol{\epsilon}\rangle_n\geq\frac{1}{16}\Vert v-\hat{\Pi}_{J}v\Vert_n^2\right)\\
&\leq \mathbb{E}\left[1_{\mathcal{E}_{\delta,J\cup J_0}}\exp\left(-\frac{n}{2^9}\frac{\Vert v-\hat{\Pi}_{J}v\Vert_n^2}{\sigma^2}\right)\right]\\
&\leq \exp\left(-\frac{c_\delta}{2^{10}}\frac{n(1-\rho_{q^*}^2)\kappa_l}{\sigma^2}\right),
\end{align*}
which bounds the first term in \eqref{eq:secte}.
If $\mathcal{A}$ holds, then
\begin{equation*}\|(\hat{\Pi}_{J}-\hat{\Pi}_{J_0})(f-v)\|_n^2\leq 4\|f-v\|_n^2\leq 8c'(1-\rho_{q^*}^2)\kappa\leq 8c'(1-\rho_{q^*}^2)\kappa_l.
\end{equation*}
Applying this and \eqref{eq:pyt2}, we obtain 
\begin{align*}
&\mathbb{P}\left(\mathcal{E}_{\delta,J\cup J_0}\cap\mathcal{A}\cap \langle (\hat{\Pi}_{J}-\hat{\Pi}_{J_0})(f-v),\boldsymbol{\epsilon}\rangle_n\geq\frac{1}{16}\Vert v-\hat{\Pi}_{J}v\Vert_n^2\right)\\
&\leq \mathbb{P}\left(\mathcal{E}_{\delta,J\cup J_0}\cap\mathcal{A}\cap \langle (\hat{\Pi}_{J}-\hat{\Pi}_{J_0})(f-v),\boldsymbol{\epsilon}\rangle_n\geq\frac{1}{32}\frac{(1-\delta)^2}{(1+\delta)}(1-\rho_{q*}^2)\kappa_l\right)\\
&\leq\exp\left(-\frac{c_\delta^2}{2^{14}c'}\frac{n(1-\rho_{q*}^2)\kappa_l}{\sigma^2}\right)
\end{align*}
which bounds the second term in \eqref{eq:secte}.
This completes the proof.\qed

\section*{Acknowledgements}
Finally, I sincerely would like to thank Prof. Enno Mammen for his support during the preparation of this paper.

\bibliographystyle{plain}
\bibliography{lit}

\begin{thebibliography}{10}

\bibitem{ACL}
E.~Arias-Castro and K.~Lounici.
\newblock Estimation and variable selection with exponential weights.
\newblock {\em Electron. J. Stat.}, 8:328--354, 2014.

\bibitem{BDDW}
R.~Baraniuk, M.~Davenport, R.~DeVore, and M.~Wakin.
\newblock A simple proof of the restricted isometry property for random
  matrices.
\newblock {\em Constr. Approx.}, 28:253--263, 2008.

\bibitem{BBM}
A.~Barron, L.~Birgé, and P.~Massart.
\newblock Risk bounds for model selection via penalization.
\newblock {\em Probab. Theory Related Fields}, 113:301--413, 1999.

\bibitem{BvG}
P.~Bühlmann and S.~van~de Geer.
\newblock {\em Statistics for High-Dimensional Data. Methods, Theory and
  Applications}.
\newblock Springer, Heidelberg, 2011.

\bibitem{BRT}
P.~J. Bickel, Y.~Ritov, and A.~B. Tsybakov.
\newblock Simultaneous analysis of lasso and dantzig selector.
\newblock {\em Ann. Statist.}, 37:1705--1732, 2009.

\bibitem{BM}
L.~Birgé and P.~Massart.
\newblock Minimum contrast estimators on sieves: exponential bounds and rates
  of convergence.
\newblock {\em Bernoulli}, 4:329--375, 1998.

\bibitem{CT}
E.~Candès and T.~Tao.
\newblock The dantzig selector: Statistical estimation when $p$ is much larger
  than $n$.
\newblock {\em Ann. Statist.}, 35:2313--2351, 2007.

\bibitem{CT2}
E.~J. Candès and T.~Tao.
\newblock Near-optimal signal recovery from random projections: universal
  encoding strategies?
\newblock {\em IEEE Trans. Inform. Theory}, 52:5406--5425, 2006.

\bibitem{CD}
L.~Comminges and A.~S. Dalalyan.
\newblock Tight conditions for consistency of variable selection in the context
  of high dimensionality.
\newblock {\em Ann. Statist.}, 40:2667--2696, 2012.

\bibitem{DIT}
A.~Dalalyan, Y.~Ingster, and A.~B. Tsybakov.
\newblock Statistical inference in compound functional models.
\newblock {\em Probab. Theory Related Fields}, 158:513--532, 2014.

\bibitem{ra10}
M.~Fornasier and H.~Rauhut.
\newblock Compressive sensing.
\newblock In {\em Handbook of Mathematical Methods in Imaging}, pages 187--228.
  Springer, 2011.

\bibitem{GI}
G.~Gayraud and Y.~Ingster.
\newblock Detection of sparse additive functions.
\newblock {\em Electron. J. Stat.}, 6:1409--1448, 2012.

\bibitem{HHW}
J.~Huang, J.~L. Horowitz, and F.~Wei.
\newblock Variable selection in nonparametric additive models.
\newblock {\em Ann. Statist.}, 38:2282--2313, 2010.

\bibitem{KW}
S.~Kayalar and H.~L. Weinert.
\newblock Error bounds for the method of alternating projections.
\newblock {\em Math. Control Signals Systems}, 1:43--59, 1988.

\bibitem{K}
H.~Kober.
\newblock A theorem on banach spaces.
\newblock {\em Compositio Math.}, 7:135--140, 1939.

\bibitem{Ko}
V.~Koltchinskii.
\newblock {\em Oracle Inequalities in Empirical Risk Minimization and Sparse
  Recovery Problems}.
\newblock Springer, Heidelberg, 2011.

\bibitem{KY}
V.~Koltchinskii and M.~Yuan.
\newblock Sparsity in multiple kernel learning.
\newblock {\em Ann. Statist.}, 38:3660--3695, 2010.

\bibitem{LM}
B.~Laurent and P.~Massart.
\newblock Adaptive estimation of a quadratic functional by model selection.
\newblock {\em Ann. Statist.}, 28:1302--1338, 2000.

\bibitem{M}
P.~Massart.
\newblock {\em Concentration Inequalities and Model Selection}.
\newblock Springer, Berlin, 2007.

\bibitem{MGB}
L.~Meier, S.~van~de Geer, and P.~Bühlmann.
\newblock High-dimensional additive modeling.
\newblock {\em Ann. Statist.}, 37:3779--3821, 2009.

\bibitem{RWY}
G.~Raskutti, M.~J. Wainwright, and B.~Yu.
\newblock Minimax-optimal rates for sparse additive models over kernel classes
  via convex programming.
\newblock {\em J. Mach. Learn. Res.}, 13:389--427, 2012.

\bibitem{R}
H.~Rauhut.
\newblock Compressive sensing and structured random matrices.
\newblock In {\em Theoretical Foundations and Numerical Methods for Sparse
  Recovery}, Radon Ser. Comput. Appl. Math., 9, pages 1--92. Walter de Gruyter,
  Berlin, 2010.

\bibitem{RTs}
P.~Rigollet and A.~Tsybakov.
\newblock Exponential screening and optimal rates of sparse estimation.
\newblock {\em Ann. Statist.}, 39:731--771, 2011.

\bibitem{RV2}
M.~Rudelson and R.~Vershynin.
\newblock On sparse reconstruction from fourier and gaussian measurements.
\newblock {\em Comm. Pure Appl. Math.}, 61:1025--1045, 2008.

\bibitem{SuSu}
T.~Suzuki and M.~Sugiyama.
\newblock Fast learning rate of multiple kernel learning: trade-off between
  sparsity and smoothness.
\newblock {\em Ann. Statist.}, 41:1381--1405, 2013.

\bibitem{Tal}
M.~Talagrand.
\newblock New concentration inequalities in product spaces.
\newblock {\em Invent. Math.}, 126:505--563, 1996.

\bibitem{T}
A.~B. Tsybakov.
\newblock {\em Introduction to Nonparametric Estimation}.
\newblock Springer, New York, 2009.

\bibitem{VW}
A.~W. {Van der Vaart} and J.~A. Wellner.
\newblock {\em Weak convergence and empirical processes. With applications to
  statistics}.
\newblock Springer, New York, 1996.

\bibitem{W}
M.~Wahl.
\newblock A theory of nonparametric regression in the presence of complex
  nuisance components.
\newblock preprint. available at http://arxiv.org/abs/1403.1088, 2014.

\bibitem{W4}
M.~Wahl.
\newblock Optimal conditions for support recovery in additive {G}aussian white
  noise models.
\newblock preprint, 2015.

\bibitem{Wa}
M.~J. Wainwright.
\newblock Information-theoretic limits on sparsity recovery in the
  high-dimensional and noisy setting.
\newblock {\em IEEE Trans. Inform. Theory}, 55:5728--5741, 2009.

\end{thebibliography}
\end{document}